\pdfoutput=1
\RequirePackage[l2tabu, orthodox]{nag}

\documentclass[reqno]{amsart}

\usepackage{lmodern}
\usepackage[T1]{fontenc}
\usepackage[utf8]{inputenc}
\usepackage[english]{babel}
\usepackage{microtype} 

\usepackage{amsmath,amssymb,amsthm,amscd,mathrsfs,latexsym,mathtools,mathdots,enumerate,tikz,bm,url}
\usetikzlibrary{arrows,matrix,decorations.pathmorphing,decorations.markings}

\usepackage[enableskew,vcentermath]{youngtab}
\usepackage{ytableau}

\usetikzlibrary{positioning}
\usepackage[enableskew,vcentermath]{youngtab}
\usepackage[capitalize]{cleveref}

\usepackage{todonotes}
\presetkeys%
    {todonotes}%
    {inline,backgroundcolor=yellow}{}

\usepackage{graphicx}

\usepackage{ifpdf}
\graphicspath{{./figures/}}
\usepackage{epstopdf}
\DeclareGraphicsExtensions{.png,.jpg,.eps,.epsf}

\newtheorem{theorem}{Theorem}
\newtheorem{proposition}[theorem]{Proposition}
\newtheorem{lemma}[theorem]{Lemma}
\newtheorem{corollary}[theorem]{Corollary}

\theoremstyle{definition}
\newtheorem{definition}[theorem]{Definition}
\newtheorem{example}[theorem]{Example}
\newtheorem{remark}[theorem]{Remark}
\newtheorem{question}[theorem]{Question}

\newcommand{\defin}[1]{\emph{#1}}

\newcommand{\setN}{\mathbb{N}}
\newcommand{\setZ}{\mathbb{Z}}

\newcommand{\setC}{\mathbb{C}}

\newcommand{\xvec}{\mathbf{x}}

\newcommand{\tpi}{\tilde{\pi}}
\newcommand{\ttheta}{\tilde{\theta}}
\newcommand{\trho}{\tilde{\rho}}

\newcommand{\symS}{S}

\newcommand{\revsort}[1]{\boldsymbol{\lambda}(#1)}

\newcommand{\elementaryE}{\mathrm{e}}
\newcommand{\atom}{\mathcal{A}}

\newcommand{\hlPolyP}{\mathrm{P}}

\newcommand{\key}{\mathcal{K}}

\newcommand{\macdonaldP}{\mathrm{P}}
\newcommand{\macdonaldE}{\mathrm{E}}
\newcommand{\schurS}{\mathrm{s}}

\newcommand{\NAF}{\mathrm{NAF}}

\DeclareMathOperator{\length}{\ell}

\DeclareMathOperator{\id}{id}

\DeclareMathOperator{\leg}{leg}
\DeclareMathOperator{\arm}{arm}
\DeclareMathOperator{\inv}{inv}
\DeclareMathOperator{\maj}{maj}
\DeclareMathOperator{\coinv}{coinv}

\DeclareMathOperator{\someStat}{stat}

\DeclareMathOperator{\Des}{\mathrm{Des}}

\title{Properties of non-symmetric Macdonald polynomials at $q=1$ and $q=0$}

\author{Per Alexandersson and Mehtaab Sawhney}
\address{Dept. of Mathematics, Royal Institute of Technology, SE-100 44 Stockholm, Sweden}
\address{Dept. of Mathematics, Massachusetts Institute of Technology, Cambridge MA, U.S.A}
\email{per.w.alexandersson@gmail.com}
\email{msawhney@mit.edu}



\begin{document}

\begin{abstract}
We examine the non-symmetric Macdonald polynomials $\macdonaldE_\lambda(\xvec;q,t)$ at $q=1$, 
as well as the more general permuted-basement Macdonald polynomials.
When $q=1$, we show that $\macdonaldE_\lambda(\xvec;1,t)$ is symmetric and independent of $t$
whenever $\lambda$ is a partition.
Furthermore, we show that for general $\lambda$, this expression factors into a symmetric and a non-symmetric part,
where the symmetric part is independent of $t$, while the non-symmetric part only depends on the relative order of the entries in $\lambda$.

We also examine the case $q=0$, which give rise to so called permuted-basement $t$-atoms.
We prove expansion-properties of these, and as a corollary, prove that Demazure characters (key polynomials)
expand positively into permuted-basement atoms. This complements the result that permuted-basement atoms are atom-positive.
Finally, we show that a product of a permuted-basement atom and a Schur polynomial is again 
positive in the same permuted-basement atom basis, and thus interpolates between two results by
Haglund, Luoto, Mason and van Willigenburg.

The common theme in this project is the application of basement-permuting operators as well as combinatorics on fillings,
by applying results in a previous article by the first author.
\end{abstract}

\maketitle

\tableofcontents

\section{Introduction}


The non-symmetric Macdonald polynomials, $\macdonaldE_\lambda(\xvec;q,t)$
were introduced by Macdonald and Opdam in \cite{Macdonald1994,Opdam1995}. 
They can be defined in other root systems. We only consider the type $A$
for which there is a combinatorial rule, discovered by Haglund, Haiman and Loehr, \cite{HaglundHaimanLoehr2005}.
These non-symmetric Macdonald polynomials specialize to the Demazure characters, $\mathcal{K}_\lambda$, (or key polynomials)
at $q=t=0$, and at $t=0$, they are affine Demazure characters, see \cite{Ion2003}.
Furthermore, at $q=t=\infty$, $\macdonaldE_\lambda(\xvec;\infty,\infty)$ reduces to the so called
Demazure atoms, $\mathcal{A}_\lambda$, (also known as standard bases), see \cite{Mason2009,Lascoux1990Keys}.
The stable limit of $\macdonaldE_\lambda(\xvec;q,t)$ gives the classical symmetric Macdonald polynomials 
(up to a rational function in $q$ and $t$, depending on $\lambda$),
denoted $P_\lambda(\xvec;q,t)$, see \cite{Macdonald1995}.
For a quick overview, see the diagram \eqref{eq:speciOverview} below, where $\ast$ denotes this stable limit.

\begin{align}\label{eq:speciOverview}
\begin{CD}
&& \mathcal{A}_\lambda(\xvec) \\
& &  @AA{ \substack{ q=\infty \\ t=\infty } }A \\
\mathcal{K}_\lambda(\xvec) @<{q=t=0}<< \macdonaldE_\lambda(\xvec;q,t) @>{\ast}>> \macdonaldP_\lambda(\xvec;q,t) \\
  & & @VV{ \substack{ \lambda \text{ partition}\\ q=1 \\ t=0  } }V  @VV{ \substack{ q=1 \\ t=0  } }V \\
   && e_{\lambda'}(\xvec) & & m_\lambda(\xvec).
\end{CD}
\end{align}

The topic of this paper is a generalization that arise naturally from Haglund's combinatorial formula,
namely the \emph{permuted basement Macdonald polynomials}, see \cite{Alexandersson15gbMacdonald, Ferreira2011}.
Recently, an alcove walk model was given for these as well, see \cite{FeiginMakedonskyi2015b,FeiginMakedonskyi2015}. 
This generalize the alcove walk model by Ram and Yip, \cite{RamYip2011}, for general type non-symmetric Macdonald polynomials.

The permuted basement Macdonald polynomials are indexed with an extra parameter, $\sigma$,
which is a permutation. For each fixed $\sigma \in \symS_n$, the set $\{ \macdonaldE^\sigma_\lambda(\xvec;q,t)\}_\lambda$
is a basis for the polynomial ring $\setC[x_1,\dotsc,x_n]$, as $\lambda$ ranges over weak compositions of length $n$.

The current paper is the only one (to our knowledge) that studies this property in the permuted-basement setting.
There has been previous research regarding various factorization properties of Macdonald polynomials,
for example, \cite{Descouens2008,Descouens2012} concern symmetric Macdonald 
polynomials and the modified Madonald polynomials when $t$ is taken to be a root of unity.
In \cite{1707.00897}, various factorization properties 
of non-symmetric Macdonald polynomials are observed experimentally (in particular, the specialization $q=u^{-2}$, $t=u$) 
in the last section of the article.

\subsection{Main results}

The first part of the paper concerns properties of the specialization $\macdonaldE^\sigma_\lambda(\xvec;1,t)$.
We show that for any fixed basement $\sigma$ and composition $\lambda$,
\begin{align}\label{eq:introFactorizationResult}
 \macdonaldE^\sigma_\lambda(\xvec;1,t) = \left( \elementaryE_{\lambda'}(\xvec) / \elementaryE_{\mu'}(\xvec) \right) \macdonaldE^\sigma_{\mu}(\xvec;1,t)
\end{align}
where $\mu$ is the \emph{weak standardization} (defined below) of $\lambda$. 
Note that $\elementaryE_{\lambda'}(\xvec) / \elementaryE_{\mu'}(\xvec)$
is an elementary symmetric polynomial \emph{independent} of $t$.
We also show that in the case $\lambda$ is a \emph{partition}, we have
\[
 \macdonaldE^\sigma_\lambda(\xvec;1,t) = \elementaryE_{\lambda'}(\xvec)
\]
which is independent of $\sigma$ and $t$. This property is rather surprising and not evident from
Haglund's combinatorial formula, \cite{HaglundNonSymmetricMacdonald2008}.
Our proofs mainly use properties of Demazure--Lusztig operators, see \eqref{eq:demLuzOps} below for the definition.

\bigskip 

In the second half of the paper, we study properties of the specialization $\macdonaldE^\sigma_\lambda(\xvec;0,t)$.
At $\sigma = \id$, these are $t$-deformations of so called \emph{Demazure atoms},
so it is natural to introduce the notation $\atom^{\sigma}_\alpha(\xvec;t) \coloneqq \macdonaldE^\sigma_\lambda(\xvec;0,t)$,
which are referred to as \emph{$t$-atoms}.
The $t$-atoms for $\sigma=\id$ were initially considered in \cite{Haglund2011463},
where they prove a close connection with Hall--Littlewood polynomials.
The Hall--Littlewood polynomials $\hlPolyP_\lambda(\xvec;t)$ are obtained as the specialization $q=0$ in the 
classical Macdonald polynomial $\macdonaldP_\lambda(\xvec;q,t)$.
In fact, it was proven in \cite{Haglund2011463} that the ordinary Hall--Littlewood polynomials $\hlPolyP_\mu(\xvec;t)$ can be expressed as
\begin{equation}\label{eq:HLinGenAtoms}
\hlPolyP_\mu(\xvec;t) = \sum_{\gamma : \revsort{\gamma}=\mu} \atom_\gamma(\xvec;t)
\end{equation}
whenever $\mu$ is a partition, and $\revsort{\gamma}$ denotes the unique partition with the parts of $\gamma$
rearranged in decreasing order.

\medskip

Our main result regarding the $t$-atoms is as follows:
If $\tau \geq \sigma$ in Bruhat order, then $\atom^{\tau}_\gamma(\xvec;t)$ admits the expansion
\begin{align}\label{eq:texpansion}
\atom^{\tau}_\gamma(\xvec;t) = \sum_{\alpha : \revsort{\alpha} = \revsort{\gamma} } c^{\tau\sigma}_{\gamma\alpha}(t) \atom^{\sigma}_\alpha(\xvec;t)
\end{align}
where the $c^{\tau\sigma}_{\gamma\alpha}(t)$
are polynomials in $t$, with the property that $c^{\tau\sigma}_{\gamma\alpha}(t) \geq 0$ whenever $0\leq t \leq 1$.
\medskip

\cref{eq:texpansion} is a generalization of the fact that key polynomials 
and permuted-basement atoms expand positively into Demazure atoms,
see \emph{e.g.} \cite{Pun2016Thesis,Mason2008}.
Letting $t=0$, we obtain the general result that whenever $\tau \geq \sigma$ in Bruhat order,
\begin{align}\label{eq:atomexpansion}
\atom^\tau_\gamma(\xvec) = \sum_{\alpha : \revsort{\alpha} = \revsort{\gamma} } c^{\tau\sigma}_{\gamma\alpha} \atom^{\sigma}_\alpha(\xvec),\qquad  \text{where}\qquad c^{\tau\sigma}_{\gamma\alpha}\in \{0,1\}. 
\end{align}
\cref{fig:positivity} below illustrates how various bases of polynomials are related 
under expansion. 
We prove the dashed relations \eqref{eq:texpansion} and \eqref{eq:atomexpansion} in this paper.
In the figure, we have the permuted-basement atoms, $\atom^{\tau}_\gamma(\xvec) \coloneqq \atom^{\tau}_\gamma(\xvec;0)$,
the key polynomials $\key_\gamma(\xvec) \coloneqq \atom^{\omega_0}_\gamma(\xvec)$
and the Demazure atoms $\atom_\gamma(\xvec) \coloneqq \atom^{\id}_\gamma(\xvec)$.
Finally, $\omega_0$ denotes the longest permutation (in $\symS_n$).

\begin{figure}[!ht]
\centering

\begin{tikzpicture}[xscale=3.6,yscale=2.3,scale=0.8, every node/.style={scale=0.8}]
\tikzset{
    vertex/.style = {
        draw,
	align=left,
        outer sep = 2pt,
        inner sep = 3pt,
	minimum height = 1cm
    },
    pluses/.style={
	dashed, decoration={markings,
	mark=between positions 1.5pt and 1 step 6pt with {
	\draw[-] (0,1.5pt) -- (0,-1.5pt);
       }
    },
    postaction=decorate,
  },
  posExp/.style = {thick,->,black},
  superSet/.style = {double equal sign distance,-implies},
}
\node[vertex] (atom)		at ( 1, 0) {Demazure atom $\atom_\alpha(\xvec)$};
\node[vertex] (key)		at ( 1, 2) {Key $\key_\alpha(\xvec)$};
\node[vertex] (generalAtom)	at ( 1, 1) {Perm. bas. atom $\atom^\sigma_\alpha(\xvec)$};
\node[vertex] (generalTAtom)	at ( 3, 1) {Perm. bas. $t$-atom $\atom^\tau_\alpha(\xvec;t)$};
\node[vertex] (generalTAtom2)	at ( 3, 0) {Perm. bas. $t$-atom $\atom^\sigma_\alpha(\xvec;t)$};
\node[vertex] (schurS)		at ( 2, 3) {Schur $\schurS_\lambda(\xvec)$};
\node[vertex] (hallLittlewoodP)	at ( 3, 2) {Hall--Littlewood, $\hlPolyP_\lambda(\xvec;t)$};
\draw[superSet]  (schurS) to node[left] {subset} (key);
\draw[posExp,dashed]  (key) to node[left] {$\in \{0,1\}$} (generalAtom);
\draw[posExp]  (schurS) to node[right] {$\in \setN[t]$ (Kostka--Foulkes)} (hallLittlewoodP);
\draw[posExp]  (generalAtom) to node[left] {$\in \{0,1\}$} (atom);
\draw[posExp]  (hallLittlewoodP) to node[right] {$\in \{0,1,t,t^2,\dotsc\}$} (generalTAtom);
\draw[posExp,dashed]  (generalTAtom) to node[right] {$\substack{\in \setZ[t] \\ \text{pos. when } 0\leq t\leq 1}$} (generalTAtom2);
\draw[posExp,bend left,looseness=2.2]  (key) to (atom);
\end{tikzpicture}

\caption{This graph shows various families of polynomials. 
The arrows indicate ``expands positively in'' which means that the coefficients are
either non-negative numbers or polynomials in $t$ with non-negative coefficients. 
Here, $\tau \geq \sigma$ in Bruhat order, and Schur polynomials should be interpreted as polynomials in $n$
variables \emph{or} symmetric functions depending on context.
}\label{fig:positivity}
\end{figure}
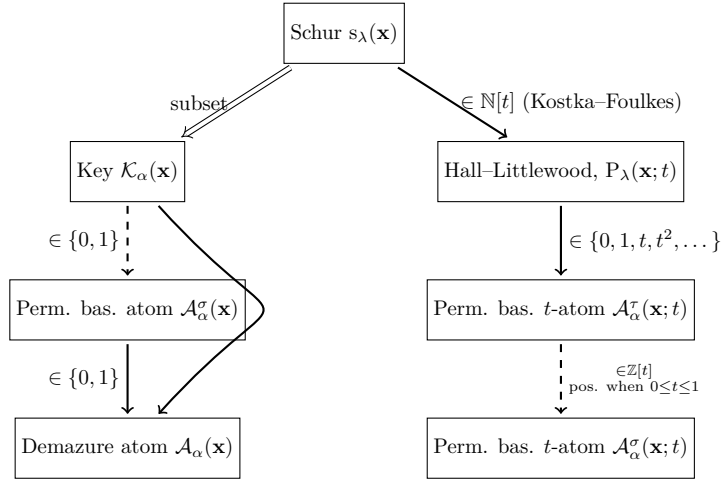
As a final corollary, by taking $\tau=\omega_0$, 
we see that key polynomials expand positively into permuted-basement Demazure atoms:
\begin{align}\label{eq:keyexpansion}
\key_\gamma(\xvec) = \sum_{\alpha : \revsort{\alpha} = \revsort{\gamma} } c^{\sigma}_{\gamma\alpha} \atom^{\sigma}_\alpha(\xvec),\qquad  \text{where}\qquad c^{\sigma}_{\gamma\alpha}\in \{0,1\}. 
\end{align}

\section{Preliminaries}

We now give the necessary background on the combinatorial model for the
permuted basement Macdonald polynomials.
The notation and some of the preliminaries is taken from \cite{Alexandersson15gbMacdonald}.

\medskip

Let $\sigma = (\sigma_1,\dots,\sigma_n)$ be a list of $n$ different positive integers and
let $\lambda=(\lambda_1,\dots,\lambda_n)$ be a \emph{weak integer composition}, that is, a vector with non-negative integer entries.
An \defin{augmented filling} of shape $\lambda$ and \defin{basement} $\sigma$
is a filling of a Young diagram of shape $(\lambda_1,\dotsc,\lambda_n)$ with positive integers,
augmented with a zeroth column filled from top to bottom with $\sigma_1,\dotsc,\sigma_n$.
Note that we use \emph{English notation} rather than the skyline fillings used in \cite{HaglundNonSymmetricMacdonald2008,Mason2009}.

\begin{definition}
Let $F$ be an augmented filling. Two boxes $a$, $b$ are \defin{attacking}
if $F(a)=F(b)$ and the boxes are either in the same column,
or they are in adjacent columns with the rightmost box in a row strictly below the other box.
\begin{align*}
\begin{ytableau}
a  \\
\none[\scriptstyle\vdots] \\
b  \\
\end{ytableau}\quad
\text{or}\quad
\begin{ytableau}
 a & \none \\
 \none[\scriptstyle\vdots] \\
  & b \\
\end{ytableau}
\end{align*}
\end{definition}
A filling is \defin{non-attacking} if there are no attacking pairs of boxes.

\begin{definition}
A \defin{triple of type $A$} is an arrangement of boxes, $a$, $b$, $c$
located such that $a$ is immediately to the left of $b$, and $c$ is somewhere below $b$,
and the row containing $a$ and $b$ is at least as long as the row containing $c$.

Similarly, a \defin{triple of type $B$} is an arrangement of boxes, $a$, $b$, $c$,
located such that $a$ is immediately to the left of $b$, and $c$ is somewhere above $a$,
and the row containing $a$ and $b$ is \emph{strictly} longer than the row containing $c$.

A type $A$ triple is an \defin{inversion triple} if the entries ordered increasingly
form a \emph{counter-clockwise} orientation. Similarly, a type $B$ triple is an inversion triple
if the entries ordered increasingly form a \emph{clockwise} orientation.
If two entries are equal, the one with the largest subscript in \cref{eq:invTriplets}
is considered to be largest.
\begin{equation}\label{eq:invTriplets}
\text{Type $A$:}\quad
\ytableausetup{centertableaux,boxsize=1.2em}
\begin{ytableau}
 a_3 & b_1 \\
 \none  & \none[\scriptstyle\vdots] \\
\none & c_2 \\
\end{ytableau}
\qquad
\text{Type $B$:}\quad
\ytableausetup{centertableaux,boxsize=1.2em}
\begin{ytableau}
c_2 & \none \\
\none[\scriptstyle\vdots]  & \none \\
a_3 & b_1 \\
\end{ytableau}
\end{equation}
\end{definition}

If $u = (i,j)$ let $d(u)$ denote $(i,j-1)$.
A \defin{descent} in $F$ is a non-basement box $u$ such that $F(d(u)) < F(u)$.
The set of descents in $F$ is denoted by $\Des(F)$.

\begin{example}
Below is a non-attacking filling of shape $(4,1,3,0,1)$ with basement $(4,5,3,2,1)$.
The bold entries are descents and the underlined entries form a type $A$ inversion triple.
There are $7$ inversion triples (of type $A$ and $B$) in total.
\[
\begin{ytableau}
\underline{4} & \underline{2} & 1 & \textbf{2} & 4\\
5 & 5\\
3 & 3 & \textbf{4} & 3\\
2\\
1 & \underline{1} \\
\end{ytableau}
\]
\end{example}

\medskip

The \defin{leg} of a box, denoted by $\leg(u)$, in an augmented diagram is the number of boxes to the right of $u$ in the diagram.
The \defin{arm} of a box $u = (r,c)$, denoted by $\arm(u)$, in an augmented diagram $\lambda$ is defined as the cardinality of
the sets
\begin{align*}
\{ (r', c) \in \lambda : r < r' \text{ and } \lambda_{r'} \leq \lambda_r \} \text{ and } \\
\{ (r', c-1) \in \lambda : r' < r \text{ and } \lambda_{r'} < \lambda_r \}.
\end{align*}

We illustrate the boxes $x$ and $y$ (in the first and second set in the union, respectively) contributing to $\arm(u)$ below.
The boxes marked $l$ contribute to $\leg(u)$.
The $\arm$ values for all boxes in the diagram are shown in the diagram on the right.
\begin{equation*}
 \begin{ytableau}
\;  & y &   &   &  \\
  & y\\
 \\
  &   & \mathbf u & l & l  & l \\
  &   & x &  \\
  &  \\
  &   & x &   &  \\
 \end{ytableau}
\qquad\qquad
 \begin{ytableau}
\;  & 4 & 2 & 2 & 1\\
  & 1\\
 \\
  & 6 & 4 & 3 & 2 & 1\\
  & 3 & 1 & 0\\
  & 1\\
  & 4 & 3 & 1 & 1\\
 \end{ytableau}
\end{equation*}
The \defin{major index}, $\maj(F)$, of an augmented filling $F$ is given by
\begin{align*}
\maj(F) = \sum_{ u \in \Des(F) } \leg(u)+1.
\end{align*}
The \defin{number of inversions}, denoted by $\inv(F)$, of a filling is the number of inversion triples of either type.
The number of \defin{coinversions}, $\coinv(F)$, is the number of type $A$ and type $B$ triples which are \emph{not}
inversion triples.

Let $\NAF_\sigma(\lambda)$ denote all non-attacking fillings of shape $\lambda$
with basement $\sigma \in \symS_n$ and entries in $\{1,\dotsc,n\}$.
\begin{example}
The set $\NAF_{3124}(1,1,0,2)$ consists of the following augmented fillings:
\begin{align*}
\substack{\young(31,12,2,443)\\ \coinv: 1\\ \maj: 1} \quad
\substack{\young(31,12,2,444)\\ \coinv: 1\\ \maj: 1} \quad
\substack{\young(32,11,2,443)\\ \coinv: 0\\ \maj: 0} \quad
\substack{\young(32,11,2,444)\\ \coinv: 0\\ \maj: 0}\\
\substack{\young(33,11,2,442)\\ \coinv: 1\\ \maj: 0} \quad
\substack{\young(33,11,2,444)\\ \coinv: 0\\ \maj: 0} \quad
\substack{\young(33,12,2,441)\\ \coinv: 2\\ \maj: 1} \quad
\substack{\young(33,12,2,444)\\ \coinv: 0\\ \maj: 1}
\end{align*}
\end{example}

Recall, the length of a permutation, $\length(\sigma)$, is the number of inversions in $\sigma$.
We let $\omega_0$ denote the unique longest permutation in $\symS_n$.
Furthermore, given an augmented filling $F$, the \emph{weight} of $F$
is the composition $\mu_1,\mu_2,\dots,$ such that $\mu_i$ is the number of non-basement entries in $F$
that are equal to $i$. We then let $\xvec^F$ be a shorthand for the product $\prod_i x_i^{\mu_i}$.

\medskip

\begin{definition}[Combinatorial formula]
Let $\sigma \in \symS_n$ and let $\lambda$ be a weak composition with $n$ parts.
The \defin{non-symmetric permuted basement Macdonald polynomial} $\macdonaldE^\sigma_\lambda(\xvec;q,t)$ is defined as
\begin{equation}\label{eq:nonSymmetricMacdonaldBasement}
\macdonaldE^\sigma_\lambda(\xvec; q,t) \coloneqq \sum_{ F \in \NAF_\sigma(\lambda)} \xvec^F q^{\maj F} t^{\coinv F} \!\!\!
\prod_{ \substack{ u \in F \\ F(d(u))\neq F(u) }} \!\!\! \frac{1-t}{1-q^{1+\leg u} t^{1+\arm u}},
\end{equation}
where $F(d(u)) \neq F(u)$ in the product index if $u$ is a box in the basement.
\end{definition}
When $\sigma = \omega_0$, we recover
the non-symmetric Macdonald polynomials defined in \cite{HaglundNonSymmetricMacdonald2008}, $\macdonaldE_\lambda(\xvec;q,t)$.

Note that the number of variables we work over is always finite and implicit from the context.
For example, if $\sigma \in \symS_n$, then $\xvec \coloneqq (x_1,\dotsc,x_n)$ in $\macdonaldE^\sigma_\lambda(\xvec; q,t)$,
and it is understood that $\lambda$ has $n$ parts.

\subsection{Bruhat order, compositions and operators}

If $\omega \in \symS_n$ is a permutation, we can decompose $\omega$ as a product $\omega = s_1s_2\dotsm s_k$
of elementary transpositions, $s_i = (i,i+1)$. 
When $k$ is minimized, $s_1s_2\dotsm s_k$ is a \defin{reduced word} of $\omega$,
and $k$ is the \defin{length} of $\omega$, which we denote by $\length(\omega)$.

The \defin{strong order} on permutations in $\symS_n$ is a partial order
defined via the cover relations that $u$ covers $v$ if $(a,b) u = v$ and $\length(u) + 1 = \length(v)$ for some transposition $(a,b)$.
The \defin{Bruhat order} is defined in a similar fashion, where only elementary transpositions are allowed in the covering relations.
We illustrate these partial orders in \cref{fig:weakAndBruhatOrder}.

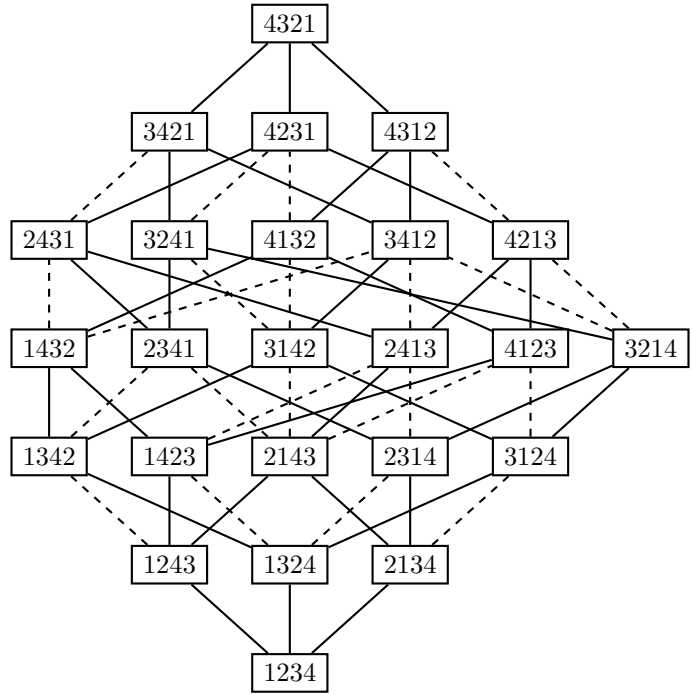
\begin{figure}[!ht]
\begin{tikzpicture}[thick,scale=0.4,yscale=-0.9,
every node/.style = {
shape= rectangle,
draw,                    
minimum width  = 1cm,
minimum height = 0.5cm,
align          = center,
text           = black},
black edge/.style  = { -,
thick,
black,
shorten >= 4pt}
]
\node (n1234) at (0.,24.) {1234};
\node (n1243) at (-4.,20.) {1243};
\node (n1324) at (0.,20.) {1324};
\node (n1342) at (-8.,16.) {1342};
\node (n1423) at (-4.,16.) {1423};
\node (n1432) at (-8.,12.) {1432};
\node (n2134) at (4.,20.) {2134};
\node (n2143) at (0.,16.) {2143};
\node (n2314) at (4.,16.) {2314};
\node (n2341) at (-4.,12.) {2341};
\node (n2413) at (4.,12.) {2413};
\node (n2431) at (-8.,8.) {2431};
\node (n3124) at (8.,16.) {3124};
\node (n3142) at (0.,12.) {3142};
\node (n3214) at (12.,12.) {3214};
\node (n3241) at (-4.,8.) {3241};
\node (n3412) at (4.,8.) {3412};
\node (n3421) at (-4.,4.) {3421};
\node (n4123) at (8.,12.) {4123};
\node (n4132) at (0.,8.) {4132};
\node (n4213) at (8.,8.) {4213};
\node (n4231) at (0.,4.) {4231};
\node (n4312) at (4.,4.) {4312};
\node (n4321) at (0.,0.) {4321};
\draw(n1234)-- (n1243);
\draw(n1234)-- (n1324);
\draw(n1234)-- (n2134);
\draw[dashed] (n1243)-- (n1342);
\draw(n1243)-- (n1423);
\draw(n1243)-- (n2143);
\draw(n1324)-- (n1342);
\draw[dashed] (n1324)-- (n1423);
\draw[dashed] (n1324)-- (n2314);
\draw(n1324)-- (n3124);
\draw(n1342)-- (n1432);
\draw[dashed] (n1342)-- (n2341);
\draw(n1342)-- (n3142);
\draw(n1423)-- (n1432);
\draw[dashed] (n1423)-- (n2413);
\draw(n1423)-- (n4123);
\draw[dashed] (n1432)-- (n2431);
\draw[dashed] (n1432)-- (n3412);
\draw(n1432)-- (n4132);
\draw(n2134)-- (n2143);
\draw(n2134)-- (n2314);
\draw[dashed] (n2134)-- (n3124);
\draw[dashed] (n2143)-- (n2341);
\draw(n2143)-- (n2413);
\draw[dashed] (n2143)-- (n3142);
\draw[dashed] (n2143)-- (n4123);
\draw(n2314)-- (n2341);
\draw[dashed] (n2314)-- (n2413);
\draw(n2314)-- (n3214);
\draw(n2341)-- (n2431);
\draw(n2341)-- (n3241);
\draw(n2413)-- (n2431);
\draw[dashed] (n2413)-- (n3412);
\draw(n2413)-- (n4213);
\draw[dashed] (n2431)-- (n3421);
\draw(n2431)-- (n4231);
\draw(n3124)-- (n3142);
\draw(n3124)-- (n3214);
\draw[dashed] (n3124)-- (n4123);
\draw[dashed] (n3142)-- (n3241);
\draw(n3142)-- (n3412);
\draw[dashed] (n3142)-- (n4132);
\draw(n3214)-- (n3241);
\draw[dashed] (n3214)-- (n3412);
\draw[dashed] (n3214)-- (n4213);
\draw(n3241)-- (n3421);
\draw[dashed] (n3241)-- (n4231);
\draw(n3412)-- (n3421);
\draw(n3412)-- (n4312);
\draw(n3421)-- (n4321);
\draw(n4123)-- (n4132);
\draw(n4123)-- (n4213);
\draw[dashed] (n4132)-- (n4231);
\draw(n4132)-- (n4312);
\draw(n4213)-- (n4231);
\draw[dashed] (n4213)-- (n4312);
\draw(n4231)-- (n4321);
\draw(n4312)-- (n4321);
\end{tikzpicture}
\caption{The Bruhat order and strong order on $\symS_4$.
Permutations expressed in one-line notation and solid lines correspond to elementary transposition.}
\label{fig:weakAndBruhatOrder}
\end{figure}

\medskip 

Given a composition $\alpha$, let $\revsort{\alpha}$ be the unique integer partition where the parts of $\alpha$
has been rearranged in decreasing order. For example, $\revsort{2,0,1,4,9} = (9,4,2,1,0)$.
We can act with permutations on compositions (and partitions) by permutation of the entries:
\[
 \omega(\lambda) = (3,0,1,5) \text{ if } \omega=(2,4,3,1) \text{ and } \lambda = (5,3,1,0),
\]
where $\omega$ is given in one-line notation.

\bigskip

In order to prove the main result of this paper we rely heavily on the Knop--Sahi recurrence,
\emph{basement permuting operators}, and \emph{shape permuting operators}.
The Knop--Sahi recurrence relations for Macdonald polynomials \cite{Knop1997,Sahi1996} is given by the relation
\begin{equation}\label{eq:knopRelation}
 \macdonaldE_{\hat{\lambda}}(\xvec;q,t) = q^{\lambda_1} x_1 \macdonaldE_{\lambda}(x_2,\dotsc,x_n,q^{-1}x_1;q,t)
\end{equation}
where $\hat{\lambda} = (\lambda_2,\dotsc,\lambda_n, \lambda_1 +1)$.
Furthermore, note that the combinatorial formula implies that
\begin{align}\label{eq:factorOut}
\macdonaldE^\sigma_{(\lambda_1+1,\dotsc,\lambda_n+1)}(\xvec;q,t) = (x_1\dotsm x_n) \macdonaldE^\sigma_{(\lambda_1,\dotsc,\lambda_n)}(\xvec;q,t).
\end{align}

\medskip
We need some brief background on certain $t$-deformations of divided difference operators.
Let $s_i$ be a simple transposition on indices of variables and define
\[
 \partial_i = \frac{1-s_i}{x_i-x_{i+1}}, \quad \pi_i = \partial_i x_i, \quad \theta_i = \pi_i - 1.
\]
The operators $\pi_i$ and $\theta_i$ are used to define the key polynomials and Demazure atoms, respectively.
Now define the following $t$-deformations of the above operators:
\begin{align}\label{eq:demLuzOps}
\tpi_i(f) = (1-t)\pi_i(f) + t s_i(f)  \qquad \ttheta_i(f) = (1-t)\theta_i(f) + t s_i(f).
\end{align}
The $\ttheta_i$ are called the \defin{Demazure--Lusztig operators} generates
the affine Hecke algebra, see \emph{e.g.}~\cite{HaglundNonSymmetricMacdonald2008} (where $\ttheta_i$ correspond to $T_i$).
Note that these operators satisfy the braid relations, and that $\ttheta_i \tpi_i =  \tpi_i \ttheta_i = t$.
\begin{example}
As an example, $\ttheta_1( x_1^2 x_2 ) = (1 - t) x_1 x_2^2 + t x_1 x_2^2$.
\end{example}

With these definitions, we can now state the following two propositions which were proved in \cite{Alexandersson15gbMacdonald}:
\begin{proposition}[Basement permuting operators]\label{prop:basementPermutation}
Let $\lambda$ be a composition and let $\sigma$ be a permutation.
Furthermore, let $\gamma_i$ be the length of the row with basement label $i$, that is, $\gamma_i = \lambda_{\sigma^{-1}_i}$.

If $\length(\sigma s_i)  < \length(\sigma)$, then
\begin{align}\label{eq:atomOperatorOnBasement}
\ttheta_i \macdonaldE_\lambda^{\sigma}(\xvec;q,t) = \macdonaldE_\lambda^{\sigma s_i}(\xvec;q,t) \times
\begin{cases}
t \text{ if } \gamma_i \leq \gamma_{i+1} \\
1 \text{ otherwise.}
\end{cases}
\end{align}

Similarly, if $\length(\sigma s_i)  > \length(\sigma)$, then
\begin{align}\label{eq:keyOperatorOnBasement}
\tpi_i \macdonaldE_\lambda^{\sigma}(\xvec;q,t) = \macdonaldE_\lambda^{\sigma s_i}(\xvec;q,t) \times
\begin{cases}
 t \text{ if } \gamma_i < \gamma_{i+1} \\
 1 \text{ otherwise.}
\end{cases}
\end{align}
\end{proposition}
Consequently, we see that $\tpi_i$ and  $\ttheta_i$ move the basement up and down, respectively, in the Bruhat order.

\begin{proposition}[Shape permuting operators]\label{prop:thetaShapeTrans}
If $\lambda_j < \lambda_{j+1}$, $\sigma_j = i+1$ and $\sigma_{j+1} =i$ for some $i$, $j$,
then
\begin{equation}\label{eq:macdonaldShapeTrans}
\macdonaldE^\sigma_{s_j \lambda}(\xvec; q, t) = \left( \ttheta_i + \frac{1-t}{1-q^{1+\leg u}t^{\arm u}} \right) \macdonaldE^\sigma_{\lambda}(\xvec; q, t),
\end{equation}
where  $u = (j+1, \lambda_{j}+1)$ in the diagram of shape $\lambda$.
\end{proposition}

Note that these formulas together with the Knop--Sahi recurrence uniquely define the Macdonald polynomials recursively,
with the initial condition that for the empty composition, $\macdonaldE_{0\dotsc 0}(\xvec;q,t)=1$.

Finally, we will need the following result from \cite{Alexandersson15gbMacdonald}:
\begin{theorem}[Partial symmetry]\label{thm:partialSymmetry}
Suppose $\alpha_j = \alpha_{j+1}$ and $\{\sigma_j, \sigma_{j+1} \}$ take the values $\{i, i+1\}$ for some $j$, $i$,
then $\macdonaldE^\sigma_\alpha (\xvec;q,t)$ is symmetric in $x_i, x_{i+1}$.
\end{theorem}

\section{A basement invariance}

In this section, we prove bijectively that whenever $\lambda$ is a partition,
we have $\macdonaldE_\lambda^\sigma(\xvec;1,0) = \elementaryE_{\lambda'}(\xvec)$.
Note that this is \emph{independent} of the basement $\sigma$, which at a first glance might be surprising.

\begin{lemma}
Let $D$ be a diagram of shape $2^m 1^n$, where the first column has fixed distinct entries in $\setN$.
Furthermore, if $S\subseteq \setN$ be a set of $m$ integers
then there is a unique way of placing the entries in $S$ into the second column of $D$
such that the resulting filling has no coinversions.
\end{lemma}
\begin{proof}
We provide an algorithm for filling in the second column of the diagram. 
Begin by letting $C$ be the topmost box in the second column and let $L(C)$ to be the box to the left of $C$. 
In order to pick an entry for $C$, we do the following:

If there is an element in $S$ which is less 
than or equal to $L(C)$, remove it from $S$ let it be the value of $C$.

Otherwise, remove the maximal element in $S$ 
and let this be the value of $C$. 

Iterate this procedure for the remaining entries in the second column while moving $C$ downwards.
It is straightforward to verify that the result is coinversion-free and that every choice 
for the element in second column is forced.
\end{proof}

\begin{corollary}
If $\lambda$ is a partition with at most $n$ parts and $\sigma \in \symS_n$, then 
\[
\macdonaldE_\lambda^\sigma(\xvec;1,0) = \elementaryE_{\lambda'}(\xvec).
\] 
\end{corollary}
\begin{proof}
Fix a basement $\sigma$ and choose sets of elements for each of 
the remaining columns. 
Note that all such choices are in natural correspondence with the monomials whose sum is $\elementaryE_{\lambda'}$.
By applying the previous lemma inductively column by column,
it follows that there is a unique filling with the the specified column sets.
The combinatorial formula now implies 
that $\macdonaldE_\lambda^\sigma(\xvec;1,0) = \elementaryE_{\lambda'}(\xvec)$ as desired.
\end{proof}
We use a similar approach to give bijections among families of coinversion-free fillings 
of general composition shapes in \cite{AlexanderssonSawhney17}.

\begin{example}
Here are the nine fillings associated with $\macdonaldE_{210}^{132}(\xvec;1,0)$.
In other words, it is the set of non-attacking, coinversion-free fillings of shape $(2,1,0)$
and basement $132$.
\[
\begin{ytableau}
{1} & 1 & 1\\
{3} & 2\\
{2}\\
\end{ytableau}
\,
\begin{ytableau}
{1} & 1 & 1\\
{3} & 3\\
{2}\\
\end{ytableau}
\,
\begin{ytableau}
{1} & 1 & 2\\
{3} & 2\\
{2}\\
\end{ytableau}
\,
\begin{ytableau}
{1} & 1 & 2\\
{3} & 3\\
{2}\\
\end{ytableau}
\,
\begin{ytableau}
{1} & 1 & 3\\
{3} & 2\\
{2}\\
\end{ytableau}
\,
\begin{ytableau}
{1} & 1 & 3\\
{3} & 3\\
{2}\\
\end{ytableau}
\,
\begin{ytableau}
{1} & 3 & 1\\
{3} & 2\\
{2}\\
\end{ytableau}
\,
\begin{ytableau}
{1} & 3 & 2\\
{3} & 2\\
{2}\\
\end{ytableau}
\,
\begin{ytableau}
{1} & 3 & 3\\
{3} & 2\\
{2}\\
\end{ytableau}
\]
The sum of the weights is $x_1^2x_2 + x_1^2x_3 + \dotsb + x_2x_3^2 = \elementaryE_{210}(\xvec)$.
\end{example}

\section{The factorization property}

Let $\lambda$ be a composition. The \emph{weak standardization} of $\lambda$, denoted by $\tilde{\lambda}$,
is the lex-smallest composition such that for all $i$, $j$, we have
\[
\lambda_i \leq \lambda_j \Rightarrow \tilde{\lambda}_i \leq \tilde{\lambda}_j.
\]
For example, $\lambda = (6,2,5,2,3,3)$ gives $\tilde{\lambda} = (3,0,2,0,1,1)$.

\begin{lemma}\label{lem:onecolumn}
If $\lambda = 1^m0^n$, then $\macdonaldE^{\sigma}_{\lambda}(\xvec;1,t) = e_m(\xvec).$
\end{lemma}
\begin{proof}
We begin by showing this statement for $\sigma = \text{id}$.

Using \cref{thm:partialSymmetry},
we have that $\macdonaldE^{\text{id}}_{\lambda}(\xvec;1,t)$ is symmetric in $x_1$, \ldots, $x_m$ and symmetric in
$x_{m+1}$, \ldots, $x_{m+n}$.
Furthermore, using the combinatorial formula,
we can easily see that there is exactly one non-attacking filling of weight $\lambda$.
This filling has major index $0$. In other words,
\[
 [\xvec^\lambda] \macdonaldE^{\text{id}}_{\lambda}(\xvec;1,t) = 1.
\]
It is therefore enough show that the polynomial is symmetric in $x_m$ and $x_{m+1}$.
A result in \cite{HaglundNonSymmetricMacdonald2008}
implies that a polynomial $f$ is symmetric in $x_m,x_{m+1}$ if and only if $\tpi_m(f)=f$.
Hence, it suffices to show
\begin{align}\label{eq:piSymmetry}
 \tpi_m \macdonaldE^{\text{id}}_{\lambda}(\xvec;1,t) = \macdonaldE^{\text{id}}_{\lambda}(\xvec;1,t).
\end{align}
\cref{prop:basementPermutation} gives that
\[
\tpi_m \macdonaldE^{\text{id}}_{\lambda}(\xvec;1,t) = \macdonaldE^{s_m}_{\lambda}(\xvec;1,t).
\]
Hence, it remains to show that $\macdonaldE^{\text{id}}_{\lambda}(\xvec;1,t) = \macdonaldE^{s_m}_{\lambda}(\xvec;1,t)$.
We do this by exhibiting a weight-preserving bijection between fillings of shape $\lambda$ with identity basement,
and those with $s_m$ as basement. 
The bijection is given by simply permuting the basement labels in row $m$ and $m+1$,
since both coinversions and the non-attacking condition are preserved, so the result is a valid filling.
Finally, since $\arm(u)=0$ for the box in position $(m,1)$, it is straightforward to verify 
that the weight is preserved under this map.
\medskip 

The statement for general $\sigma$ now follows by applying the basement permuting operators $\tpi_i$ repeatedly
on both sides of the identity $\macdonaldE^{\sigma}_{\lambda}(\xvec;1,t) = e_m(\xvec)$.
The right hand side is unchanged since these operators preserve symmetric functions.
\end{proof}

\bigskip 

We say that $\lambda \leq \mu$ in the Bruhat order if there is a sequence of transpositions, $s_{i_1}\dotsm s_{i_k}$
such that $s_{i_1}\dotsm s_{i_k} \lambda = \mu$ and where each application of a transposition increases the number of inversions.
\begin{lemma}\label{lem:reduction} 
If $\lambda$ and $\mu$ are compositions such that $\lambda \leq \mu$ in the Bruhat order,
then the following implication holds:
\begin{align}
  \frac{ \macdonaldE^{w_0}_{\lambda}(\xvec;1, t) }{ \macdonaldE^{w_0}_{\tilde{\lambda}}(\xvec;1,t) } = F_\lambda(\xvec)
  \qquad \Longrightarrow \qquad
  \frac{ \macdonaldE^{w_0}_{\mu}(\xvec;1, t) }{ \macdonaldE^{w_0}_{\tilde{\mu}}(\xvec;1,t) } = F_\lambda(\xvec)
\end{align}
where $F_\lambda(\xvec)$ is any function symmetric in $x_1,\dotsc,x_n$.
\end{lemma}
\begin{proof}
It suffices to show the implication for any simple transposition, $s_i \lambda = \mu$ that increases the number of inversions.
Suppose that
\begin{align*}
  \macdonaldE^{w_0}_{\lambda}(\xvec;1, t) = F_\lambda(\xvec) \macdonaldE^{w_0}_{\tilde{\lambda}}(\xvec;1,t) 
\end{align*}
for some composition $\lambda$.
By \cref{prop:thetaShapeTrans}, we note that the shape permuting operator is the same on both sides for $q=1$.
That is, for any composition $\lambda$ with $\lambda_i < \lambda_{i+1}$ we have
\begin{align*}
\left( \ttheta_i + \frac{1-t}{1-t^{\arm u}} \right) \macdonaldE^{w_0}_{\lambda}(\xvec;1, t) = \macdonaldE^{w_0}_{s_i\lambda}(\xvec;1,t)
\end{align*}
and
\begin{align*}
\left( \ttheta_i + \frac{1-t}{1-t^{\arm u}} \right) F_\lambda(\xvec) \macdonaldE^{w_0}_{\tilde{\lambda}}(\xvec;1, t) = F_\lambda(\xvec)\macdonaldE^{w_0}_{s_i\tilde{\lambda}}(\xvec;1, t),
\end{align*}
where $\arm(u)\geq 1$ has the same value in both diagrams $\lambda$ and $\tilde{\lambda}$. 

\end{proof}

\medskip 

To simplify typesetting of the upcoming proofs, we will sometimes use the notation
\begin{align}\label{eq:compositionNotation}
\macdonaldE[ (a_1)^{b_1}, \dotsc, (a_k)^{b_k} ] \coloneqq \macdonaldE^{w_0}_\lambda(\xvec;1,t) 
\end{align}
where $\lambda$ is the composition 
\[
 (\underbrace{a_1,\dotsc,a_1}_{b_1},\underbrace{a_2,\dotsc,a_2}_{b_2},\dotsc, \underbrace{a_k,\dotsc,a_k}_{b_k}).
\]

\begin{lemma}\label{lem:multicolumn}
We have the identity
\[
\frac{  \macdonaldE[ (1)^{b_1},(2)^{b_2}, \dotsc, (k)^{b_k},(0)^{b_0} ] }{  \macdonaldE[ (0)^{b_1},(1)^{b_2}, \dotsc, (k-1)^{b_k},(0)^{b_0}] }
=e_{b_1+\ldots+b_k}(\xvec).
\]
\end{lemma}
\begin{proof}
We prove this lemma by induction on $k$, where the base case $k=1$ is given by \cref{lem:onecolumn}. 
For $k >1$, by \cref{prop:thetaShapeTrans} and a similar reasoning as in \cref{lem:reduction},
it is  enough to prove that 
\[
\frac{  \macdonaldE[ (1)^{b_1},(0)^{b_0},(2)^{b_2}, \dotsc, (k)^{b_k} ]     }{  \macdonaldE[ (0)^{b_0+b_1},(1)^{b_2}, \dotsc, (k-1)^{b_k}] }
=e_{b_1+\ldots+b_k}(\xvec).
\]
Furthermore, through repeated application of the Knop--Sahi recurrence \cref{eq:knopRelation} it suffices to prove
\[
\frac{  \macdonaldE[ (1)^{b_2}, \dotsc, (k-1)^{b_k},(1)^{b_1},(0)^{b_0}]     }{  \macdonaldE[ (0)^{b_2}, \dotsc, (k-2)^{b_k},(0)^{b_0+b_1}] }
=e_{b_1+\ldots+b_k}(\xvec).
\] 
Again using \cref{prop:thetaShapeTrans}, we reduce the above to the $k-1$ case
\[
\frac{  \macdonaldE[ (1)^{b_1+b_2}, \dotsc, (k-1)^{b_k},(0)^{b_0}]     }{  \macdonaldE[ (0)^{b_1+b_2}, \dotsc, (k-2)^{b_k},(0)^{b_0}] }
=e_{b_1+\ldots+b_k}(\xvec),
\] 
which is true by induction.
\end{proof}

\begin{proposition}\label{prop:fractionIsElementary}
If $\lambda$ is a composition, then
\begin{align}\label{eq:symFraction}
  \frac{ \macdonaldE^{w_0}_{\lambda}(\xvec;1, t) }{ \macdonaldE^{w_0}_{\tilde{\lambda}}(\xvec;1,t) } = F_\lambda(\xvec)
\end{align}
where $F_\lambda(\xvec)$ is an elementary symmetric polynomial.
\end{proposition}
\begin{proof}
We prove the proposition by induction on $|\lambda|$ and the number of inversions of $\lambda$. 
Note that the result is trivial if $|\lambda| \leq 1$.

Given $\lambda$, there are several cases to consider:

\noindent 
\textit{Case 1: $\min_i\lambda_i \geq 1$.}
The result follows by inductive hypothesis on the size of the composition
by using \cref{eq:factorOut} in the numerator.

\noindent 
\textit{Case 2: $\lambda$ is not weakly increasing.}
We can reduce this case to a composition with fewer inversions using \cref{lem:reduction}.

\noindent 
\textit{Case 3: $\lambda$ is weakly increasing.} 
It is enough to prove that 
\[
\frac{ \macdonaldE[ (a_0)^{b_0}, \dotsc, (a_k)^{b_k} ]  }{ \macdonaldE[ (0)^{b_0}, \dotsc, (k)^{b_k} ]  } 
\]
is an elementary symmetric polynomial where $0=a_0<a_1<a_2<\dotsb$. 
Using the cyclic shift relation \eqref{eq:knopRelation} in the numerator and denominator,
it suffices to show that 
\begin{align}\label{eq:fractionToFactor}
\frac{  \macdonaldE[ (a_1-1)^{b_1},(a_2-1)^{b_2}, \dotsc, (a_k-1)^{b_k}, (0)^{b_0} ]     }{  \macdonaldE[ (0)^{b_1},(1)^{b_2}, \dotsc, (k-1)^{b_k}, (0)^{b_0}] }
\end{align}
is an elementary symmetric polynomial. 
If $a_1=1$, the result follows by the inductive hypothesis on the size of the composition. 
Otherwise, by rewriting \cref{eq:fractionToFactor}, it is enough to prove that
\begin{align*}
\frac{  \macdonaldE[ (a_1-1)^{b_1},(a_2-1)^{b_2}, \dotsc, (a_k-1)^{b_k}, (0)^{b_0} ]     }{  \macdonaldE[ (1)^{b_1},(2)^{b_2}, \dotsc, (k)^{b_k}, (0)^{b_0}] }\cdot\frac{  \macdonaldE[ (1)^{b_1},(2)^{b_2}, \dotsc, (k)^{b_k}, (0)^{b_0} ]     }{  \macdonaldE[ (0)^{b_1},(1)^{b_2}, \dotsc, (k-1)^{b_k}, (0)^{b_0}] }
\end{align*}
is an elementary symmetric polynomial. 
The first fraction is an elementary symmetric polynomial by induction, since it is of the right form with a smaller size.
According to \cref{lem:reduction}, the second fraction is also an elementary symmetric polynomial.
\end{proof}

\begin{theorem}\label{thm:mainFactor}
If $\lambda$ is a composition and $\sigma \in \symS_n$, then
 \[
  \frac{ \macdonaldE^{\sigma}_\lambda(\xvec;1,t) }{ \macdonaldE^{\sigma}_{\tilde{\lambda}}(\xvec;1,t) } = F_\lambda(\xvec),
 \]
where $F_\lambda(\xvec)$ is an elementary symmetric polynomial independent of $t$. 
\end{theorem}
\begin{proof}
From \cref{prop:fractionIsElementary}, we have that
\[
 \macdonaldE^{w_0}_\lambda(\xvec;1,t) = F_\lambda(\xvec) \macdonaldE^{w_0}_{\tilde{\lambda}}(\xvec;1,t)
\]
where $F_\lambda$ is an elementary symmetric polynomial. 
Applying basement-permuting operators from \cref{prop:basementPermutation} on both sides then gives 
\[
  \macdonaldE^\sigma_\lambda(\xvec;1,t) = F_\lambda(\xvec) \macdonaldE^\sigma_{\tilde{\lambda}}(\xvec;1,t).
\]
Note that applying a basement-permuting operator might give an extra factor of $t$,
but since $\lambda_{i} \leq \lambda_{j}$ if and only if $\tilde{\lambda_{i}}\leq \tilde{\lambda_{j}}$,
these extra factors cancel.
\end{proof}

We are now ready to prove the following surprising identity,
which was first observed through computational evidence by J. Haglund and the first author.
\begin{theorem}\label{mainresult}
If $\lambda$ is a \emph{partition} and $\sigma \in \symS_n$, then
 \[
  \macdonaldE^\sigma_\lambda(\xvec;1,t) = \macdonaldE^\sigma_{{\lambda}}(\xvec;1,0)=e_{\lambda'}(\xvec). 
 \]
\end{theorem}
\begin{proof}
It is enough to prove that $\macdonaldE^{w_0}_\lambda(\xvec;1,t)=e_{\lambda'}(\xvec)$
as the more general statement follows from using \cref{prop:basementPermutation}.

By using the previous theorem, it is enough to prove that
\[
\frac{  \macdonaldE[ (k)^{b_0},\dotsc, (0)^{b_k} ]     }{  \macdonaldE[ (k-1)^{b_0}, \dotsc, (0)^{b_{k-1}+b_k}] }=\macdonaldE[ (1)^{b_0+\dotsc+b_{k-1}},(0)^{b_k}].
\]
We show this via induction on $k$. The base case $k=1$ is trivial,
so assume $k>1$ and note that repeated use of \cref{prop:thetaShapeTrans} implies that it is enough to prove
\[
\frac{  \macdonaldE[ (k-1)^{b_1},\dotsc, (0)^{b_k},(k)^{b_0} ]     }{  \macdonaldE[ (k-2)^{b_1}, \dotsc, (0)^{b_{k-1}+b_k},(k-1)^{b_0}] }=\macdonaldE[ (1)^{b_0+\dotsc+b_{k-1}},(0)^{b_k}].
\]
By using the Knop--Sahi recurrence \eqref{eq:knopRelation}, it suffices to show that 
\[
\frac{  \macdonaldE[ (k-1)^{b_0+b_1},\dotsc, (0)^{b_k} ]     }{  \macdonaldE[ (k-2)^{b_0+b_1}, \dotsc, (0)^{b_{k-1}+b_k}] }=\macdonaldE[ (1)^{b_0+\dotsc+b_{k-1}},(0)^{b_k}]
\]
which now follows from induction.
\end{proof}

\begin{corollary}
The previous proof can be extended to show that
\[
F_{\lambda}(\xvec)=\frac{e_{\lambda'}(\xvec)}{e_{(\tilde{\lambda})'}(\xvec)}
\] for partition $\lambda$.
\end{corollary}
Note that the parts of $\lambda'$ is a super-set of the parts of $(\tilde{\lambda})'$,
so the above expression is indeed some elementary symmetric polynomial.

\medskip

Our results are in some sense optimal: for general compositions $\lambda$,
it happens that $\macdonaldE^\sigma_{\tilde{\lambda}}(\xvec;1,t)$ cannot be factorized further.
For example, Mathematica computations suggest that 
\[
 \macdonaldE^{(3,1,5,2,4)}_{(0,2,3,1,0)}(\xvec;1,t) \text{ and } \macdonaldE^{(3,1,5,2,4)}_{(0,1,1,1,0)}(\xvec;1,0)
\] 
are irreducible.

\subsection{Discussion}

It is natural to ask whether or not there are bijective proofs of the identities we consider.
\begin{question}
Is there a bijective proof of the case $\sigma=\omega_0$ of \cref{thm:mainFactor} that establish
\[
\macdonaldE_{\lambda}(\xvec;1, t)= \frac{e_{\lambda'}(\xvec)}{e_{(\tilde{\lambda})'}(\xvec)} \macdonaldE_{\tilde{\lambda}}(\xvec;1,t)?
\]
\end{question}
Since a priori $\macdonaldE^\sigma_{{\lambda}}(\xvec;1,t)$ 
is only a rational function in $t$, this seems like a difficult challenge. 
We therefore pose a more conservative question:
\begin{question}
Is there a combinatorial explanation of the identity $\macdonaldE^\sigma_\lambda(\xvec;1,t) = e_{\lambda'}(\xvec)$ whenever $\lambda$
is a partition?
\end{question}

We finish this section by discussing properties of the family $\{\macdonaldE_\lambda(\xvec;1,0)\}$ as $\lambda$
ranges over compositions with $n$ parts.
It is a basis for $\setC[x_1,\dotsc,x_n]$ and naturally extends the elementary symmetric functions.
Furthermore, it is shown in \cite{AssafKostka} that $\macdonaldE_\lambda(\xvec;1,0)$
expands positively into key polynomials, where the coefficients are given by the classical Kostka coefficients.
Furthermore, $\{\macdonaldE_\lambda(\xvec;q,0)\}$ exhibit properties very similar to those of modified Hall--Littlewood polynomials.
In particular, these expand positively into key polynomials with Kostka--Foulkes polynomials (in $q$) as coefficients.
There are representation-theoretical explanations for these expansions as well, see \cite{AssafKostka,AlexanderssonSawhney17} 
and references therein for details.

It is known that a product of a Schur polynomial and a key polynomial is key-positive (see \emph{e.g.} \cref{prop:schurProdPos} below),
and thus a product of an elementary symmetric polynomial and a key polynomial is key positive.
It is therefore natural to ask if this extends to the non-symmetric elementary polynomials.
However, a quick computer search reveals that
\[
\macdonaldE_{030}(\xvec;1,0) \key_{201}(\xvec)
\]
does not expand positively into key polynomials.

\section{Positive expansions at $t=0$}

By specializing the combinatorial formula \cref{eq:nonSymmetricMacdonaldBasement} with $q=0$,
we obtain a combinatorial formula for the permuted-basement Demazure $t$-atoms.

\begin{example}
As an example, $\atom^{1423}_{2301}(x_1,x_2,x_3,x_4;t)$ is equal to
\begin{align*}
 &(1-t) t \cdot x_1^2 x_2^3 x_3 + (1-t)\cdot x_1^2 x_2^2 x_3 x_4 + (1-t)^2 \cdot x_1^2 x_2 x_3^2 x_4  + (1-t) \cdot x_1^2 x_3^3 x_4 \\
& + (1-t)\cdot x_1^2 x_2 x_3 x_4^2 + (1-t)\cdot x_1^2 x_3^2 x_4^2   + x_1^2 x_3 x_4^3
\end{align*}
where the corresponding fillings are 
\begin{align*}
&\young(111,4222,2,33), \qquad
\young(111,4422,2,33), \qquad
\young(111,4432,2,33), \qquad
\young(111,4433,2,33) \\
&\young(111,4442,2,33), \qquad
\young(111,4443,2,33), \qquad
\young(111,4444,2,33).
\end{align*}
\end{example}
\medskip 

In this section, we show how to construct permuted-basement Demazure $t$-atoms
via Demazure--Luzstig operators.
First consider \cref{prop:basementPermutation} and \cref{prop:thetaShapeTrans} at $q=0$.
Note that \cref{prop:thetaShapeTrans} simplifies, where we use the fact that $\ttheta_i +(1-t) = \tpi_i$.
Hence, the shape-permuting operator reduce to a basement-permuting operator.
This ``duality'' between shape and basement was first observed at $t=0$ in \cite{Mason2009},
where S.~Mason gave an alternative combinatorial description of key polynomials which is not immediate from
the combinatorial formula for the non-symmetric Macdonald polynomials.
A similar duality holds for general values of $t$, see \cite{Alexandersson15gbMacdonald}.
\medskip

To get a better overview of \cref{prop:basementPermutation} and \cref{prop:thetaShapeTrans}, 
we present the statements as actions on the basement and shape as follows:
\begin{example}\label{ex:operatorActions}
The operators $\tpi_i$ and $\ttheta_i$ act as follows on diagram shapes and basements.
Note that we only care about the relative order of row lengths. A box with a dot might either be present or not,
indicating weak or strict difference between row lengths.

\begin{align}
\ttheta_i\;\circ\; 
\begin{ytableau}
\scriptstyle{\mathbf{i+1}}  &  \\
\none[\vdots]  \\
\mathbf{i}  & & \\
\end{ytableau}
\;=\;
\begin{ytableau}
\mathbf{i}  &  \\
\none[\vdots]  \\
\scriptstyle{\mathbf{i+1}}  &  & \\
\end{ytableau} 
\qquad\qquad
\ttheta_i\;\circ\;
\begin{ytableau}
\scriptstyle{\mathbf{i+1}}  &  & \\
\none[\vdots]  \\
\mathbf{i}  & & \cdot \\
\end{ytableau}
\;=\;
t\times 
\begin{ytableau}
\mathbf{i}  &  & \\
\none[\vdots]  \\
\scriptstyle{\mathbf{i+1}}  &  & \cdot \\
\end{ytableau}
\label{eq:thetaoperatorBasement}
\end{align}

\begin{align}
\tpi_i\;\circ\;& 
\begin{ytableau}
\mathbf{i}  &  &  \\
\none[\vdots]  \\
\scriptstyle{\mathbf{i+1}}  &  & \cdot \\
\end{ytableau}
\;=\;
\begin{ytableau}
\scriptstyle{\mathbf{i+1}}  &  &  \\
\none[\vdots]  \\
\mathbf{i}  &  & \cdot \\
\end{ytableau}  
\qquad \qquad
\tpi_i\;\circ\;
\begin{ytableau}
\mathbf{i}  &   \\
\none[\vdots]  \\
\scriptstyle{\mathbf{i+1}}  &  & \\
\end{ytableau}
\;=\;
t\times 
\begin{ytableau}
\scriptstyle{\mathbf{i+1}}  &  \\
\none[\vdots]  \\
\mathbf{i}  &  & \\
\end{ytableau} 
\label{eq:pioperatorBasement} 
\end{align}

The operators acting on the shape can be described pictorially as
\begin{align}\label{eq:operatorshape}
\tpi_i\;\circ\;& 
\begin{ytableau}
\scriptstyle{\mathbf{i+1}}  & & \cdot \\
\mathbf{i}  & & \\
\end{ytableau}
\quad=\quad 
\begin{ytableau}
\scriptstyle{\mathbf{i+1}}  & &  \\
\mathbf{i}  &  & \cdot \\
\end{ytableau}  \qquad\qquad
\ttheta_i\;\circ\;
\begin{ytableau}
\mathbf{i}  & &  \\
\scriptstyle{\mathbf{i+1}}  & \\
\end{ytableau}
\quad=\quad 
\begin{ytableau}
\mathbf{i}  &  \\
\scriptstyle{\mathbf{i+1}}  & & \\
\end{ytableau}
\end{align}
which are easily obtained from \cref{prop:thetaShapeTrans} at $q=0$, together with the fact that $\ttheta_i \tpi_i = t$.
\end{example}

The following proposition also appeared in \cite{Alexandersson15gbMacdonald}, however
the proof we present here is different and more constructive.
\begin{proposition}
Given $\lambda$ and $\sigma$, there is a sequence $\trho_{i_1}\cdots \trho_{i_\ell}$ such that
\begin{equation}
\atom_\lambda^\sigma(\xvec;t) = \trho_{i_1}\cdots \trho_{i_\ell} \xvec^\lambda
\end{equation}
where $\lambda$ is the partition with the parts of $\lambda$ in decreasing order
and each $\trho_{i_j}$ is one of $\ttheta_i$ or $\tpi_i$.
\end{proposition}
\begin{proof}
Given $(\sigma,\lambda)$, let the number of \defin{monotone pairs}
be the number of pairs $(i,j)$ such that
\[
\sigma_i < \sigma_j \text{ and } \lambda_i < \lambda_j \qquad \text{ or } \qquad 
\sigma_i > \sigma_j \text{ and } \lambda_i \geq \lambda_j.
\]
We do induction over number of monotone pairs. First note that if there are no monotone pairs in $(\sigma,\lambda)$,
then the longest row has basement label $1$, the second longest row has basement label $2$ and so on.
It then follows that every row in a filling with basement $\sigma$ and shape $\lambda$ has to be constant,
implying that $\atom_\lambda^\sigma(\xvec;t) = \xvec^\lambda$.

Assume that there is some monotone pairs determined by $(\sigma,\lambda)$.
A permutation with at least one inversion must have a descent, and for a similar reason,
there is at least one monotone pair of the form
\[
 \begin{ytableau}
\scriptstyle{\mathbf{i+1}}  &  &  \\
\none[\vdots]  \\
\mathbf{i}  &  & \cdot \\
\end{ytableau} \\
\qquad
\text{ or }
\qquad
\begin{ytableau}
\mathbf{i}  &  \\
\none[\vdots]  \\
\scriptstyle{\mathbf{i+1}}  &  & \\
\end{ytableau}.
\]
These match the right hand sides of \eqref{eq:pioperatorBasement} and \eqref{eq:thetaoperatorBasement}.
By induction, $\atom_\lambda^\sigma(\xvec;t)$ can therefore be obtained
from some $\atom_\lambda^{\sigma s_i}(\xvec;t)$ by applying either $\tpi_i$ or $\ttheta_i$.
\end{proof}

\begin{example}
We illustrate the above proposition by expressing $\atom^{3142}_{3102}(\xvec;t)$ in terms of operators.
The shape and basement associated with this atom is given in the first augmented diagram in \eqref{eq:opexample}.
\begin{align}\label{eq:opexample}
\ytableausetup{centertableaux,boxsize=1em}
\begin{ytableau}
3 &  & &\\
1 & \\
4 \\
2 & & 
\end{ytableau}
\xleftarrow{\;\tpi_2\;}
\begin{ytableau}
2 &  & &\\
1 & \\
4 \\
3 & &
\end{ytableau}
\xleftarrow{\;\tpi_1\;}
\begin{ytableau}
1 &  & &\\
2 & \\
4 \\
3 & &
\end{ytableau}
\xleftarrow{\;\ttheta_2\;}
\begin{ytableau}
1 &  & &\\
3 & \\
4 \\
2 & &
\end{ytableau}
\end{align}
The rows with labels $2$ and $3$ constitute a monotone pair and can be obtained using \eqref{eq:pioperatorBasement},
which explains the $\tpi_2$-arrow. 
Continuing on with $\tpi_1$ followed by $\ttheta_2$
leads to an augmented diagram without any monotone pairs, so $\atom^{1342}_{3102}(\xvec;t) = \xvec^{(3,2,1,0)}$.
Finally, following the arrows yields the operator expression
\[
 \atom^{3142}_{3102}(\xvec;t) = \tpi_2 \tpi_1 \ttheta_2 \xvec^{(3,2,1,0)}.
\]

\end{example}

\begin{proposition}
If $\sigma = s_i \tau$ with $\length(\sigma)>\length(\tau)$, then
\begin{align}\label{eq:tatomexpansion}
\atom^{\sigma}_\lambda(\xvec;t) =  
\begin{cases}
\atom^{\tau}_{s_i\lambda}(\xvec;t) + t^{\someStat(\lambda,\sigma,i)}(1-t)\atom^{\tau}_\lambda(\xvec;t) &\text{ if } \lambda_{i} > \lambda_{i+1} \\
\atom^{\tau}_{s_i\lambda}(\xvec;t) &\text{ otherwise},
\end{cases}
\end{align}
where $\someStat(\lambda,\sigma,i)$ is a non-negative integer depending on $\lambda$, $\sigma$ and $i$.
\end{proposition}
\begin{proof}
We prove this statement via induction over $\length(\tau)$.

\medskip
\noindent
\textbf{Case $\tau = \id$ and $\lambda_{i} \leq \lambda_{i+1}$:} 
We need to show that $\atom^{s_i}_\lambda(\xvec;t) = \atom^{\id}_{s_i\lambda}(\xvec;t)$.
Since $\tpi_i$ is invertible, it suffices to show that
\[
\tpi_i \atom^{s_i}_\lambda(\xvec;t) = \tpi_i  \atom^{\id}_{s_i\lambda}(\xvec;t).
\]
This equality now follows from using \eqref{eq:operatorshape} on the 
left hand side and \eqref{eq:pioperatorBasement} on the right hand side.

\medskip
\noindent
\textbf{Case $\tau = \id$ and $\lambda_{i} > \lambda_{i+1}$:} It suffices to prove that
\[
 \atom^{s_i}_\lambda(\xvec;t) = \atom^{\id}_{s_i\lambda}(\xvec;t) + (1-t)\atom^{\id}_\lambda(\xvec;t).
\]
Note that the left hand side is equal to $\tpi_i \atom^{\id}_\lambda(\xvec;t)$ using \eqref{eq:pioperatorBasement},
while the left hand side is equal to $[\ttheta_i + (1-t)] \atom^{\id}_\lambda(\xvec;t)$ where we use \eqref{eq:operatorshape}.
Since $\tpi_i = [\ttheta_i + (1-t)]$, this proves the identity.

This proves the base case. The general case now follows from applying $\tpi_j$ on both sides,
thus increasing the lengths of the basements. We examine the details in the following two cases.

\medskip
\noindent
\textbf{Case $\tau \in \symS_n$ and $\lambda_{i} \leq \lambda_{i+1}$:}
Suppose $\atom^{\sigma}_\lambda(\xvec;t) = \atom^{\tau}_{s_i\lambda}(\xvec;t)$.
As diagrams, we have the equality
\[
\begin{ytableau}
\mathbf{b}  &  & \cdot \\
\mathbf{a}  & &  \\
\end{ytableau}
=
\begin{ytableau}
\mathbf{a}  &  & \\
\mathbf{b}  & & \cdot \\
\end{ytableau}
\]
for rows $i$ and $i+1$, $b>a$, while the remaining rows are identical.
If $\length(\sigma s_j) > \length(\sigma)$, we can conclude that if $a=j$, then $b \neq j+1$.
We now compare the row lengths of the rows with basement label $j$ and $j+1$
and apply the basement-permuting $\tpi_j$ from \eqref{eq:pioperatorBasement} on both sides.
Note that the row lengths that are compared are the same on both sides, meaning that
if we need \eqref{eq:pioperatorBasement} to increase the basement on the left hand side, 
the same relation acts the same way on the right hand side.
In other words, we have the implication
\[
\atom^{\sigma}_\lambda(\xvec;t) = \atom^{\tau}_{s_i\lambda}(\xvec;t) \quad \Longrightarrow \quad
\atom^{\sigma s_j}_\lambda(\xvec;t) = \atom^{\tau s_j}_{s_i\lambda}(\xvec;t)
\]
whenever $\length(\sigma s_j) > \length(\sigma)$ and $\lambda_{i} \leq \lambda_{i+1}$.

\medskip
\noindent
\textbf{Case $\tau \in \symS_n$ and $\lambda_{i} > \lambda_{i+1}$:}
Again, suppose we have the diagram identity
\[
\begin{ytableau}
\mathbf{b}  & \\
\mathbf{a}  & &  \\
\end{ytableau}
=
\begin{ytableau}
\mathbf{a}  & & \\
\mathbf{b}  &  \\
\end{ytableau}
+
t^{\someStat(\lambda,\sigma,i)}(1-t)
\begin{ytableau}
\mathbf{a}  & \\
\mathbf{b}  & & \\
\end{ytableau}
\]
for some $\lambda$, $\sigma$ and that $\length(\sigma s_j) > \length(\sigma)$.
As in the previous case, if $a=j$ then $b \neq j+1$.
If $j \notin \{a-1,a,b-1,b\}$, applying $\tpi_j$ on both sides yield the implication
\begin{align*}
\atom^{\sigma}_\lambda(\xvec;t) &= \atom^{\tau}_{s_i\lambda}(\xvec;t) + t^{\someStat(\lambda,\sigma,i)}(1-t)\atom^{\tau}_\lambda(\xvec;t)  \\
&\Longrightarrow \\
\atom^{\sigma s_j}_\lambda(\xvec;t) &=  \atom^{\tau s_j}_{s_i\lambda}(\xvec;t) + t^{\someStat(\lambda,\sigma,i)}(1-t)\atom^{\tau s_j}_\lambda(\xvec;t) 
\end{align*}
because --- depending on the relative row lengths of the rows with basement labels $j$, $j+1$ --- we either multiply each of the three terms by $t$ or not at all. 

It remains to verify the cases $j \in \{a-1,a,b-1,b\}$. 
Case by case study after applying $\tpi_j$ on both sides shows that
\[
 \atom^{\sigma s_j}_\lambda(\xvec;t) = \atom^{\tau s_j}_{s_i\lambda}(\xvec;t) + t^{\epsilon+\someStat(\lambda,\sigma,i)}(1-t)\atom^{\tau s_j}_\lambda(\xvec;t) 
\]
where (using the same notation as in \cref{prop:basementPermutation}, $\gamma_i$ being the length of the row with basement label $i$) 
\begin{itemize}
\item $\epsilon = -1$ if $j=a-1$ and $\gamma_a > \gamma_{a-1} \geq \gamma_b$,
\item $\epsilon = 1$ if $j=a$ and $\gamma_a \geq \gamma_{a+1} > \gamma_b$,
\item $\epsilon = 1$ if $j=b-1$ and $\gamma_a > \gamma_{b-1} \geq \gamma_b$,
\item $\epsilon = -1$ if $j=b$ and $\gamma_a \geq \gamma_{b+1} > \gamma_b$
\end{itemize}
and $\epsilon=0$ otherwise. Thus, we have that
\[
 \atom^{\sigma s_j}_\lambda(\xvec;t) - \atom^{\tau s_j}_{s_i\lambda}(\xvec;t) = t^{\epsilon+\someStat(\lambda,\sigma,i)}(1-t)\atom^{\tau s_j}_\lambda(\xvec;t) 
\]
where the left hand side is a polynomial. 
Furthermore, $\atom^{\tau s_j}_\lambda(\xvec;t)$ is not a multiple of $t$ --- this follows from the combinatorial formula \eqref{eq:nonSymmetricMacdonaldBasement}.
Hence, $\epsilon+\someStat(\lambda,\sigma,i)$ must be non-negative.
\end{proof}

%

\begin{corollary}
If $\tau \geq \sigma$ in Bruhat order then $\atom^{\tau}_\gamma(\xvec;t)$ admits the expansion
\[
\atom^{\tau}_\gamma(\xvec;t) = \sum_{\lambda : \revsort{\lambda} = \revsort{\gamma} } c^{\tau\sigma}_{\gamma\lambda}(t) \atom^{\sigma}_\lambda(\xvec;t)
\]
where the $c^{\tau\sigma}_{\gamma\lambda}(t)$
are polynomials in $t$, with the property that $c^{\tau\sigma}_{\gamma\lambda}(t) \geq 0$ whenever $0\leq t \leq 1$.
\end{corollary}

\begin{corollary}
If $\tau \geq \sigma$ in Bruhat order, then $\atom^{\tau}_\gamma(\xvec)$ admits the expansion
\[
\atom^{\tau}_\gamma(\xvec) = \sum_{\lambda : \revsort{\lambda} = \revsort{\gamma} } c^{\tau\sigma}_{\lambda\gamma} \atom^{\sigma}_\lambda(\xvec)
\]
where $c^\sigma_{\lambda\gamma} \in \{0,1\}$.
\end{corollary}
\begin{proof}
Let $t=0$ in \eqref{eq:tatomexpansion}. It is then clear that all coefficients are non-negative integers.
Furthermore, since key polynomials $(\tau = \omega_0)$ expands into Demazure atoms $(\sigma = \id)$
with coefficients in $\{0,1\}$, (see \emph{e.g.} \cite{Lascoux1990Keys,Mason2009}) the statement follows.
\end{proof}

In \cite{Haglund2011Refinements}, the cases $\sigma = \id$ and $\sigma=\omega_0$ of the following 
proposition were proved. We give an interpolation between these results:
\begin{proposition}\label{prop:schurProdPos}
The coefficients $d^{\mu\sigma}_{\lambda\gamma}$ in the expansion
\begin{align*}
 \schurS_\mu(\xvec) \times \atom^\sigma_\lambda(\xvec) = \sum_{\gamma} d^{\mu\sigma}_{\lambda\gamma} \atom^\sigma_\gamma(\xvec)
\end{align*}
are non-negative integers. 

Remember that $\xvec = (x_1,\dotsc,x_n)$, so we evaluate $\schurS_\mu(\xvec)$ in a finite alphabet.
\end{proposition}
\begin{proof}
With the case $\sigma=\id$ as a starting point (proved in \cite{Haglund2011Refinements}),
we can apply $\pi_i$ on both sides, ($\pi_i$ commutes with any symmetric function, in particular $\schurS_\lambda(\xvec)$),
and thus we may walk upwards in the Bruhat order and obtain the statement for any basement $\sigma$.
Note that \cref{prop:basementPermutation} implies that $\pi_i$ applied to $\atom^\sigma_\gamma(\xvec)$ either increase 
$\sigma$ in Bruhat order, or kills that term.
\end{proof}

Note that the above result implies that the products $\elementaryE_\mu \times \atom^\sigma_\lambda(\xvec)$
and $h_\mu \times \atom^\sigma_\lambda(\xvec)$ also expand non-negatively into $\sigma$-atoms.
It would be interesting to give a precise rule for this expansion, 
as well as a Murnaghan--Nakayama rule for the permuted-basement Demazure atoms.

\begin{remark}
We need to mention the paper \cite{LoBue2013}, which also concerns a different type of \emph{general Demazure atoms}.
These objects are also studied in \cite{Haglund2011Refinements}, but are in general different from ours when $\sigma \neq \id$.
In particular, the polynomial families they study are not bases for $\setC[x_1,\dotsc,x_n]$, and they are \emph{not} compatible with 
the Demazure operators. 
The authors of \cite{LoBue2013,Haglund2011Refinements} construct these families by imposing an additional restriction\footnote{What they call the type-B condition}
on Haglund's combinatorial model, which enables them to perform a type of RSK.

The introductions in the two papers mention the permuted-basement Macdonald polynomials, $\macdonaldE^\sigma_\mu(\xvec;q,t)$,
but the additional restriction breaks this connection whenever $\sigma \neq \id$. 
This fact is unfortunately hidden since they use the same 
notation $\hat{E}_\gamma$ is used for two different families of polynomials.

\end{remark}

\subsection*{Acknowledgement}

The authors would like to thank Jim Haglund insightful discussions.
The first author is funded by the \emph{Knut and Alice Wallenberg Foundation} (2013.03.07).

\bibliographystyle{amsalpha}
\bibliography{bibliography}

\newcommand{\etalchar}[1]{$^{#1}$}
\providecommand{\bysame}{\leavevmode\hbox to3em{\hrulefill}\thinspace}
\providecommand{\MR}{\relax\ifhmode\unskip\space\fi MR }
\providecommand{\MRhref}[2]{%
  \href{http://www.ams.org/mathscinet-getitem?mr=#1}{#2}
}
\providecommand{\href}[2]{#2}
\begin{thebibliography}{HLMvW11b}

\bibitem[Ale15]{Alexandersson15gbMacdonald}
Per Alexandersson, \emph{{Non-symmetric Macdonald polynomials and
  {D}emazure--{L}usztig operators}}, 1--19, arXiv:1602.05153.

\bibitem[AS17]{AlexanderssonSawhney17}
Per Alexandersson and Mehtaab Sawhney, \emph{A major-index preserving map on
  fillings}, Electronic Journal of Combinatorics \textbf{24} (2017), no.~4,
  1--30.

\bibitem[Ass17]{AssafKostka}
Sami Assaf, \emph{{Nonsymmetric Macdonald polynomials and a refinement of
  Kostka--Foulkes polynomials}}, arXiv:1703.02466 (to appear in Trans. Amer.
  Math. Soc.).

\bibitem[Bog03]{Ion2003}
Ion Bogdan, \emph{{Nonsymmetric Macdonald polynomials and Demazure
  characters}}, Duke Mathematical Journal \textbf{116} (2003), no.~2, 299--318.

\bibitem[CDL17]{1707.00897}
Laura Colmenarejo, Charles~F. Dunkl, and Jean-Gabriel Luque,
  \emph{Factorizations of symmetric {M}acdonald polynomials}, 2017,
  arXiv:1602.05153.

\bibitem[DM08]{Descouens2008}
Fran{\c{c}}ois Descouens and Hideaki Morita, \emph{Factorization formulas for
  {M}acdonald polynomials}, European Journal of Combinatorics \textbf{29}
  (2008), no.~2, 395--410.

\bibitem[DMN12]{Descouens2012}
Fran{\c{c}}ois Descouens, Hideaki Morita, and Yasuhide Numata, \emph{On a
  bijective proof of a factorization formula for {M}acdonald polynomials},
  European Journal of Combinatorics \textbf{33} (2012), no.~6, 1257--1264.

\bibitem[Fer11]{Ferreira2011}
Jeffrey~Paul Ferreira, \emph{Row-strict quasisymmetric {S}chur functions,
  characterizations of {D}emazure atoms, and permuted basement nonsymmetric
  {M}acdonald polynomials}, Ph.D. thesis, University of California Davis, 2011.

\bibitem[FM15a]{FeiginMakedonskyi2015}
Evgeny Feigin and Ievgen Makedonskyi, \emph{{Generalized Weyl modules, alcove
  paths and Macdonald polynomials}}, arxiv:1512.03254.

\bibitem[FM15b]{FeiginMakedonskyi2015b}
\bysame, \emph{{Nonsymmetric Macdonald polynomials and PBW filtration: Towards
  the proof of the Cherednik--Orr conjecture}}, Journal of Combinatorial
  Theory, Series A \textbf{135} (2015), 60--84.

\bibitem[HHL{\etalchar{+}}05]{HaglundHaimanLoehr2005}
J.~Haglund, M.~Haiman, N.~Loehr, J.~B. Remmel, and A.~Ulyanov, \emph{A
  combinatorial formula for the character of the diagonal coinvariants}, Duke
  Mathematical Journal \textbf{126} (2005), no.~2, 195--232.

\bibitem[HHL08]{HaglundNonSymmetricMacdonald2008}
James Haglund, Mark Haiman, and Nick Loehr, \emph{{A} {C}ombinatorial {F}ormula
  for {N}onsymmetric {M}acdonald {P}olynomials}, American Journal of
  Mathematics \textbf{130} (2008), no.~2, 359--383.

\bibitem[HLMvW11a]{Haglund2011463}
James Haglund, Kurt~W. Luoto, Sarah Mason, and Stephanie van Willigenburg,
  \emph{{Quasisymmetric Schur functions}}, Journal of Combinatorial Theory,
  Series A \textbf{118} (2011), no.~2, 463--490.

\bibitem[HLMvW11b]{Haglund2011Refinements}
\bysame, \emph{{Refinements of the Littlewood--Richardson Rule}}, Trans. Amer.
  Math. Soc. \textbf{363} (2011), 1665--1686.

\bibitem[Kno97]{Knop1997}
Friederich Knop, \emph{Integrality of two variable {K}ostka functions}, Journal
  für die reine und angewandte Mathematik \textbf{482} (1997), 177--190.

\bibitem[LR13]{LoBue2013}
Janine LoBue and Jeffrey~B. Remmel, \emph{A {M}urnaghan--{N}akayama {R}ule for
  {G}eneralized {D}emazure {A}toms}, 2013, (Proceedings of the {FPSAC} 2013
  {C}onference held in {P}arice, {F}rance), pp.~969--980.

\bibitem[LS90]{Lascoux1990Keys}
Alain Lascoux and Marcel-Paul Sch{\"{u}}tzenberger, \emph{Keys \& standard
  bases}, Invariant theory and tableaux ({M}inneapolis, {MN}, 1988), IMA Vol.
  Math. Appl., vol.~19, Springer, New York, 1990, pp.~125--144. \MR{1035493
  (91c:05198)}

\bibitem[Mac95a]{Macdonald1994}
I.~G. Macdonald, \emph{Affine {H}ecke algebras and orthogonal polynomials},
  Séminaire Bourbaki \textbf{37} (1994-1995), 189--207 (eng).

\bibitem[Mac95b]{Macdonald1995}
\bysame, \emph{{Symmetric functions and {H}all polynomials}}, second ed.,
  {Oxford Mathematical Monographs}, The Clarendon Press Oxford University
  Press, New York, 1995, With contributions by A. Zelevinsky, Oxford Science
  Publications. \MR{MR1354144 (96h:05207)}

\bibitem[Mas08]{Mason2008}
Sarah Mason, \emph{A decomposition of {S}chur functions and an analogue of the
  {R}obinson-{S}chensted-{K}nuth algorithm}, Seminaire Lotharingien de
  Combinatoire (2008), no.~B57e.

\bibitem[Mas09]{Mason2009}
\bysame, \emph{An explicit construction of type {A} {D}emazure atoms}, Journal
  of Algebraic Combinatorics \textbf{29} (2009), no.~3, 295--313.

\bibitem[Opd95]{Opdam1995}
Eric~M. Opdam, \emph{Harmonic analysis for certain representations of graded
  {H}ecke algebras}, Acta Math. \textbf{175} (1995), no.~1, 75--121.

\bibitem[Pun16]{Pun2016Thesis}
Anna Pun, \emph{On {D}ecomposition of the {P}roduct of {D}emazure {A}toms and
  {D}emazure {C}haracters}, Ph.D. thesis, University of Pennsylvania, 2016.

\bibitem[RY11]{RamYip2011}
Arun Ram and Martha Yip, \emph{{A combinatorial formula for Macdonald
  polynomials}}, Advances in Mathematics \textbf{226} (2011), no.~1, 309--331.

\bibitem[Sah96]{Sahi1996}
Siddhartha Sahi, \emph{Interpolation, integrality, and a generalization of
  {M}acdonald's polynomials}, Internat. Math. Res. Notices \textbf{1996}
  (1996), no.~10, 457.

\end{thebibliography}

\end{document}